\documentclass[review]{elsarticle}
\usepackage{graphicx}
\usepackage{amsfonts}
\usepackage{amsmath}
\usepackage{amssymb}
\usepackage{subfigure}
\usepackage{lineno}

\journal{Journal of \LaTeX\ Templates}

\newtheorem{lemma}{Lemma}[section]

\newtheorem{theorem}{Theorem}[section]

\newtheorem{corollary}{Corollary}[section]
\newtheorem{proposition}{Proposition}[section]

\begin{document}
	
	\begin{frontmatter}
		
		\title{On a problem of Henning and Yeo about the transversal number of uniform linear systems whose 2-packing number is fixed}

		\author{Carlos A. Alfaro\fnref{1}}
		\fntext[1]{This work was partially supported by SNI and CONACyT}
		\address{Banco de M\'exico}
		\ead{carlos.alfaro@banxico.org.mx}

		\author{Adri\'an V\'azquez-\'Avila\fnref{1}}
		\address{Subdirecci\'on de Ingenier\'ia y Posgrado\\Universidad Aeron\'autica en Quer\'etaro}
		\ead{adrian.vazquez@unaq.edu.mx}
		
		
		
		

\begin{abstract}
 A linear system is a pair $(P,\mathcal{L})$ where
 $\mathcal{L}$ is a family of subsets on a ground finite set $P$ such that $|l\cap l^\prime|\leq 1$, for every $l,l^\prime \in \mathcal{L}$. If all elements of $\mathcal{L}$ of a linear system $(P,\mathcal{L})$, then the linear system is called $r$-uniform linear system. The transversal number of a linear system $(P,\mathcal{L})$, $\tau(P,\mathcal{L})$, is the minimum cardinality of a subset $\hat{P}\subseteq P$ satisfying $l\cap\hat{P}\neq\emptyset$, for every $l\in\mathcal{L}$. The 2-packing number of a linear system $(P,\mathcal{L})$, $\nu_2(P,\mathcal{L})$, is the maximum cardinality of a subset $R\subseteq\mathcal{L}$ such that, any three elements of $R$ don't have a common point (are triplewise disjoint), that is, if three elements are chosen in $R$, then they are not incidents in a common point. 
 
 For $r\geq2$, let $(P,\mathcal{L})$ be an $r$-uniform linear system. In "{\sc M. A. Henning and A. Yeo:}  {\it Hypergraphs with large transversal number,} Discrete Math. {\bf 313} (2013), no. 8, 959--966." Henning and Yeo state the following question: Is it true that if $(P,\mathcal{L})$ is an $r$-uniform linear system then $\tau(P,\mathcal{L})\leq\displaystyle\frac{|P|+|\mathcal{L}|}{r+1}$ holds for all $r\geq2$?. In this note, we give some results of $r$-uniform linear systems, whose 2-packing number is fixed, satisfying the inequality. 

\end{abstract}


\end{frontmatter}

\section{Introduction}\label{sec:intro}

A \emph{linear system} is a pair $(P,\mathcal{L})$ where
$\mathcal{L}$ is a family of subsets on a ground finite set $P$ such that $|l\cap l^\prime|\leq 1$, for every pair of distinct subsets $l,l^\prime \in \mathcal{L}$. The linear system $(P,\mathcal{L})$ is \emph{intersecting}, if $|l\cap l^\prime|=1$, for every pair of distinct subsets $l,l^\prime \in \mathcal{L}$. The elements of $P$ and $\mathcal{L}$ are called \emph{points} and \emph{lines}, respectively. An \emph{$r$-line} is a line containing exactly $r$ points. An \emph{$r$-uniform} linear system $(P,\mathcal{L})$ is a linear system such that all lines of $\mathcal{L}$ are $r$-lines. In this context, a \emph{simple graph} is a 2-uniform linear system. Throughout this paper, we will only consider linear systems of rank $r\geq2$.

Let $(P,\mathcal{L})$ be a linear system and $p\in P$ be a point. The set of lines incident to $p$ is denoted by $\mathcal{L}_p$. The \emph{degree} of $p$ is defined as $deg(p)=|\mathcal{L}_p|$ and the \emph{maximum degree} overall points of the linear systems is denoted by $\Delta=\Delta(P,\mathcal{L})$. Two points $p,q\in P$ are \emph{adjacent}, if there is a line $l\in\mathcal{L}$ such that $\{p,q\}\subseteq l$.

A {\emph{linear subsystem}} $(P^{\prime },\mathcal{L}^{\prime })$ of
a linear system $(P,\mathcal{L})$ is a linear system such that, any line $l^\prime\in\mathcal{L}^\prime$ there is a line $l\in\mathcal{L}$ satisfying $l^\prime=l\cap P^\prime$. The \emph{linear subsystem induced} by a set of lines $\mathcal{L}^{\prime}\subseteq \mathcal{L}$ is the linear subsystem $(P^{\prime },\mathcal{L}^{\prime })$, where
$P^{\prime}=\bigcup_{l\in \mathcal{L}^{\prime }} l$. The linear subsystem $(P^\prime,\mathcal{L}^\prime)$ of $(P,\mathcal{L})$ is called \emph{spanning linear subsystem}, if $P^\prime=P$. Given a linear system $(P,\mathcal{L})$ and a point $p\in P$, the linear system obtained from $(P,\mathcal{L})$ by \emph{deleting point $p$} is the linear subsystem $(P^{\prime },\mathcal{L}^{\prime })$ induced by $\mathcal{L}^{\prime }=\{l\setminus \{p\}: l\in \mathcal{L}\}$. On the other hand, given a linear system $(P,\mathcal{L})$ and a line $l\in \mathcal{L}$, the linear system obtained from $(P,\mathcal{L})$ by \emph{deleting the line $l$} is the linear subsystem $(P^{\prime },\mathcal{L}^{\prime })$ induced by $\mathcal{L}^{\prime }=\mathcal{L}\setminus \{l\}$. Finally, let $(P^{\prime},\mathcal{L}^{\prime })$ and $(P,\mathcal{L})$ be two linear systems. We say that $(P^{\prime},\mathcal{L}^{\prime })$ and $(P,\mathcal{L})$ are \emph{isomorphic}, $(P^{\prime },\mathcal{L}^{\prime })\simeq(P,\mathcal{L})$, if after of deleting points of degree 1 or 0 from both, the systems $(P^{\prime},\mathcal{L}^{\prime })$ and $(P,\mathcal{L})$ are isomorphic as hypergraphs, see \cite{MR3727901}.  

Let $(P,\mathcal{L})$ be a linear system. A subset $T\subseteq P$ is a \emph{transversal} of $(P,\mathcal{L})$ (also called \emph{vertex cover} or \emph{hitting set}), if $T\cap l\neq\emptyset$, for every line $l\in\mathcal{L}$. The minimum cardinality of a transversal of a linear system $(P,\mathcal{L})$, $\tau=\tau(P,\mathcal{L})$, is called \emph{transversal number} of $(P,\mathcal{L})$. A subset $R\subseteq \mathcal{L}$ is called a \emph{2-packing} of $(P,\mathcal{L})$, if any three elements are chosen in $R$, then they are not incidents in a common point. The maximum cardinality of a 2-packing of $(P,\mathcal{L})$, $\nu_2=\nu_2(P,\mathcal{L})$, is called \emph{2-packing number} of $(P,\mathcal{L})$. This new parameter has been studied in some papers, see for example \cite{CGCA,CCA,MR3727901,Avila3,Avila2,Avila_dom_gra}. 

In \cite{MR3727901}, was proved a relationship between the transversal and the 2-packing numbers 
\begin{equation}\label{desigualdad}
\lceil \nu_{2}/2\rceil\leq\tau
\leq \frac{\nu_2(\nu_2-1)}{2}.
\end{equation}
Hence, the transversal number of any linear system is upper bounded for a quadratic function of their 2-packing number. For some linear systems the transversal number is bounded above by a linear function of their 2-packing number, see \cite{CGCA,CCA,MR3727901}. In \cite{MR3021333}, was proved, using probabilistic methods, the
existence of $k$-uniform linear systems $(P,\mathcal{L})$ for
infinitely many $k$'s and $n=|P|$ large enough, which transversal
number is $\tau=n-o(n)$. This $k$-uniform linear
systems has 2-packing number upper bounded by $\frac{2n}{k}$.

There are also works where the transversal number of an $r$-uniform linear systems (and in more general context) is bounded above by a function of their number points and lines, see for example \cite{Dorfling,HENNING2013959}. In \cite{HENNING2013959}, was stated the fo\-llowing question: Is it true that if $(P,\mathcal{L})$ is an $r$-uniform linear system, then $\tau\leq\displaystyle\frac{|P|+|\mathcal{L}|}{r+1}$, for all $r\geq2$?. In \cite{VChvCMcD92}, was proved the inequality when $r\in\{2,3\}$. In \cite{Dorfling}, was proved the inequality when $\Delta\leq2$ and there are only two families of $r$-uniform linear systems that achieve equality in the bound. On the other hand, in \cite{Dorfling}, was given a better upper bound for the transversal number in terms of the number of points and the number of lines, namely, it was proved if $(P,\mathcal{L})$ is an $r$-uniform linear system with $\Delta\leq2$ and $r\geq3$ odd integer such that $r(r^2-3)\tau\leq(r-2)(r+1)n+(r-1)^2m+r-1$. Similar bounds was proven in the case when $r\geq2$ is even integer. 

This paper is organized as follows: In Section \ref{sec:Examples}, we present an infinite family of $r$-uniform linear systems $(P,\mathcal{L})$ which satisfy $\tau=\displaystyle\frac{|P|+|\mathcal{L}|}{r+1}$, where $r\geq3$ is an odd integer. This family of linar systems was defined in \cite{CGCA}. In Section \ref{sec:intersecting}, we prove that, if $(P,\mathcal{L})$ is an intersecting $r$-uniform linear system with $\tau=r$, then $\tau\leq\displaystyle\frac{|P|+|\mathcal{L}|}{r+1}$. In Section \ref{sec:nu_2}, we prove, if $(P,\mathcal{L})$ is an $r$-uniform linear systems with $\nu_2\in\{2,3,4\}$, then $\tau\leq\displaystyle\frac{|P|+|\mathcal{L}|}{r+1}$. Finally, in Section \ref{sec:Delta=2}, we prove that, if $(P,\mathcal{L})$ is an $r$-uniform linear systems with $\Delta=2$, then $\tau\leq\displaystyle\frac{|P|+|\mathcal{L}|}{r+1}$, satisfying the equality, if and only, if $(P,\mathcal{L})$ is an $(\nu_2-1)$-uniform linear systems with $\nu_2\geq2$ being even integer. This results was obteined first in \cite{Dorfling}. 
\section{Examples of linear systems $(P,\mathcal{L})$ with $\tau=\displaystyle\frac{|P|+|\mathcal{L}|}{r+1}$}\label{sec:Examples}

Let $(\Gamma,+)$ be an additive Abelian group, with neutral element $e$, satisfying $\sum_{g\in\Gamma}g=e$ and $2g\neq e$, for all $g\in\Gamma\setminus\{e\}$. An example of this type of groups is $(\mathbb{Z}_n,+)$, with $n\geq3$ odd integer.

Let $n=2k+1$, with $k$ a positive integer; and let $(\Gamma,+)$ be an additive Abelian group of order $n$ as above. In \cite{CGCA}, was defined the following linear system $\mathcal{C}_{n,n+1}=(P_n,\mathcal{L}_n)$, where  $$P_n=(\Gamma\times\Gamma\setminus\{e\})\cup\{p,q\}\quad\!\!\!\!\mbox{and}\quad\!\!\!\!\mathcal{L}_n=\mathcal{L}\cup\mathcal{L}_p\cup\mathcal{L}_q,$$with $$\mathcal{L}=\{L_g:g\in\Gamma\setminus\{e\}\},\mbox{ and } L_g=\{(h,g):h\in\Gamma\},$$for $g\in\Gamma\setminus\{e\}$, and $$\mathcal{L}_{p}=\{l_{p_g}:g\in\Gamma\},\mbox{ with } l_{p_g}=\{(g,h):h\in\Gamma\setminus\{e\}\}\cup\{p\},$$ for $g\in\Gamma$, and $\mathcal{L}_q=\{l_{q_g}:g\in\Gamma\}$, and $$l_{q_g}=\{(h,f_g(h)):h\in
\Gamma, f_g(h)=h+g \mbox{ with $f_g(h)\neq e$}\}\cup\{q\},$$ for $g\in\Gamma$. 

The set of lines $\mathcal{L}$ is a set of pairwise disjoint lines with $|\mathcal{L}|=n-1$ and each line of $\mathcal{L}$ has $n$ points. The set of lines $\mathcal{L}_p$ and $\mathcal{L}_q$ are lines incidents to $p$ and $q$, respectively, with $|\mathcal{L}_p|=|\mathcal{L}_p|=n$, and each line of $\mathcal{L}_p\cup\mathcal{L}_q$ has $n$ points. This linear system is an  $n$-uniform linear system with $n(n-1)+2$ points and $3n-1$ lines. Moreover, this linear system has 2 points of degree $n$ (points $p$ and $q$) and $n(n-1)$ points of degree $3$. In \cite{CGCA}, was proved the following:
\begin{theorem}\cite{CGCA}\label{thm:paper_anterior}
	The linear system $\mathcal{C}_{n,n+1}$ satisfies	
	$$\tau(\mathcal{C}_{n,n+1})=\nu_2(\mathcal{C}_{n,n+1})=n+1.$$
\end{theorem}
As a consequence of Theorem \ref{thm:paper_anterior}, we have 
\begin{corollary}
Let $(P,\mathcal{L})$ be an $r$-uniform linear system such that $(P,\mathcal{L})\simeq\mathcal{C}_{n,n+1}$, where $r\geq n$, then $\tau\leq\displaystyle\frac{|P|+|\mathcal{L}|}{r+1}$. Moreover, the equality holds, if and only if, $(P,\mathcal{L})=\mathcal{C}_{n,n+1}$.
\end{corollary}
\textit{Proof.}
Let $(P,\mathcal{L})$ be an $r$-uniform linear system such that $(P,\mathcal{L})\simeq\mathcal{C}_{n,n+1}$. Then $|P|=n(n-1)+2+k|\mathcal{L}|$, where $n+k=r$ with $k\geq0$, and $|\mathcal{L}|=3n-1$. Hence
\begin{eqnarray*}
\frac{|P|+|\mathcal{L}|}{r+1}&=&\frac{n(n-1)+2+(3n-1)(k+1)}{n+k+1}\\&=&\frac{(n-1)(n+k+1)+2(n(k+1)+1)}{n+k+1}\\&=&(n+1)+\frac{2k(n-1)}{n+k+1}\\
&\geq&n+1.	
\end{eqnarray*}
Hence, by Theorem \ref{thm:paper_anterior}, we have $\tau\leq\displaystyle\frac{|P|+|\mathcal{L}|}{n+1}$. The equality holds, if and only if, $k=0$, that is, if and only if $(P,\mathcal{L})=\mathcal{C}_{n,n+1}$.\qed

\begin{theorem}\label{thm:main}
If $r\geq2$ is an positive integer and $(P,\mathcal{L})$ is an $r$-uniform linear system with $\Delta\geq\nu_2-1$, $|\mathcal{L}|\geq\nu_2+\Delta-2$ and $\Delta\geq3$, then $\nu_2-1\leq\displaystyle\frac{|P|+|\mathcal{L}|}{r+1}$.
\end{theorem}
\textit{Proof.}
Since $|P|\geq\Delta(r-1)+1$, we have
\begin{eqnarray*}
	\frac{|P|+|\mathcal{L}|}{r+1}&\geq&\frac{\Delta(r-1)+1+\nu_2+\Delta-2}{r+1}\\
	&=&\frac{r\Delta+\nu_2-1}{r+1}\geq\frac{(\nu_2-1)(r+1)}{r+1}\\
	&=&\nu_2-1.
\end{eqnarray*}
And the statement holds.\qed	

\begin{corollary}\label{coro:main_1}
If $r\geq2$ is an positive integer and $(P,\mathcal{L})$ is an $r$-uniform linear system with $\Delta\geq\nu_2$, $|\mathcal{L}|\geq\nu_2+\Delta-1$ and $\Delta\geq3$, then $\nu_2\leq\displaystyle\frac{|P|+|\mathcal{L}|}{r+1}$.	
\end{corollary}
\textit{Proof.} The proof of this corollary is analogous to the proof of Theorem \ref{thm:main}.\qed

\section{Intersecting $r$-uniform linear systems}\label{sec:intersecting}
Through this paper, all linear systems $(P,\mathcal{L})$ are considered with $|\mathcal{L}|>\nu_2$ due to the fact $|\mathcal{L}|=\nu_2$, if and only if, $\Delta\leq2$. 

In \cite{CGCA}, was proved that any linear system $(P,\mathcal{L})$ with ``few'' lines satisfies $\tau\leq\nu_2-1$.

\begin{theorem}\cite{CGCA}\label{thm:casicucarachageneral}
Let $(P,\mathcal{L})$ be a linear system with $p,q\in P$ be two points such that $deg(p)=\Delta$ and $deg(q)=\max\{deg(x): x\in P\setminus\{p\}\}$. If $|\mathcal{L}|\leq \Delta+deg(q)+\nu_2-3$, then $\tau\leq\nu_2-1$.
\end{theorem}

If $(P,\mathcal{L})$ is an intersecting $r$-uniform linear system, then $\tau\leq r$ and $\Delta\leq r$. The proof of Lemma \ref{lemma:Delta=r} and Lemma \ref{lemma:puntos_projectivo} are analogous to the proof of Lemma 2.4 and Lemma 2.5, given in \cite{Dong}.

\begin{lemma}\label{lemma:Delta=r}
Let $(P,\mathcal{L})$ be an intersecting $r$-uniform linear system, with $r\geq3$. If $\tau=r$, then every line of $(P,\mathcal{L})$ has at most one point of degree two and $\Delta=r$.
\end{lemma}
\text{Proof.}\cite{Dong}
First, we show that every line of $\mathcal{L}$ has at most one point of degree two. Let us suppose that there a line $l\in\mathcal{L}$ and two points $p,q\in l$, such that $deg(p)=deg(q)=2$, which implies $(\mathcal{L}_p\cap\mathcal{L}_q)\setminus\{l\}=\{x\}$. Hence, the following set $\hat{T}=(l\setminus\{p,q\})\cup\{x\}$ is a transversal of $(P,\mathcal{L})$ of cardinality $r-1$, which is a contradiction. Therefore, every line of $(P,\mathcal{L})$ has at most one double point.
	
On the other hand, let us suppose that $\Delta=r-a$, where $a\geq1$. Then $|P|\leq\tau\Delta=r(r-a)$. Let $p\in P$ such that $\Delta=deg(p)$, then $|N(p)|=\Delta(r-1)=(r-a)(r-1)$, where $N(p)=\{q\in P\setminus\{p\}:\{p,q\}\in l\in\mathcal{L}_p\}$. Since $P\setminus N(p)$ is a transversal of $(P,\mathcal{L})$, then $\tau\leq|P\setminus N(p)|=r(r-a)-(r-a)(r-1)=r-a$, which is a contradiction. Therefore, $\Delta=r$.\qed
	
\begin{lemma}\label{lemma:puntos_projectivo}
	Let $(P,\mathcal{L})$ be an intersecting $r$-uniform linear system, with $r\geq3$. If $\tau=r$, then $3(r-1)|\leq|\mathcal{L}|\leq r^2-r+1$ and $|P|=r^2-r+1$.
\end{lemma}
\textit{Proof.}\cite{Dong}
	Since $(P,\mathcal{L})$ is an intersecting $r$-uniform linear system, we have $|\mathcal{L}|=\displaystyle\sum_{p\in l}deg(p)-(r-1)$. Hence, by Lemma \ref{lemma:Delta=r}, we have $3(r-1)\leq|\mathcal{L}|\leq r^2-r+1$.
	
	To show that $|P|=r^2-r+1$, let $p\in P$ such that $\Delta=deg(p)$. Then $|P|\geq|N(p)\cup\{p\}|\geq\Delta(r-1)+1=r^2-r+1$. Let us suppose that $|P|\geq r^2-r+2$, then there exists $q\in P\setminus\{N(p)\cup\{p\}\}$, such that $l_q\cap l_p\neq\emptyset$, for every $l_q\in\mathcal{L}_q$ and $l_q\in\mathcal{L}_p$. Hence, we have $|l_q|\geq deg(p)+1=r+1$, which is a contradiction. Therefore $|P|\leq r^2-r+1$. \qed
	    
The proof of Lemma \ref{lemma:nu_2=r+1} and Lemma \ref{lemma:nu_2=r}, are analogous to the proof of Lemma 4.1, given in \cite{Avila3}, and Lemma 6, given in \cite{Avila2}, respectively. 

\begin{lemma}\label{lemma:nu_2=r+1}
Let $(P,\mathcal{L})$ be an intersecting $r$-uniform linear system, with $r\geq3$ be an odd integer. If $\tau=r$, then $\nu_2=r+1$.
\end{lemma}
\textit{Proof.}\cite{Avila3}
Let $(P,\mathcal{L})$ be an intersecting $r$-uniform linear system. Since $\nu_2-1\leq|l|$, for any $l\in\mathcal{L}$, then $\nu_2\leq r+1$.	
Let $p,q\in P$ be two points such that $deg(p)=\Delta$ and $\Delta^\prime=deg(q)=\max\{deg(x): x\in
P\setminus\{p\}\}$. By Theorem \ref{thm:casicucarachageneral}, if $|\mathcal{L}|\leq\Delta+\Delta^\prime+\nu_2-3\leq3(r-1)+1$, then $\tau\leq\nu_2-1\leq r$, which implies that $\nu_2=r+1$, since $\tau=r$. On the other hand, if $|\mathcal{L}|\geq3(r-1)+2$ and $\tau=\nu_2$, then by Theorem \ref{thm:casicucarachageneral} we have $|\mathcal{L}|\geq\Delta+\Delta^\prime+\nu_2-2\geq3(r-1)+1$, which implies that $\nu_2\geq r+1$. Therefore $\nu_2=r+1$.\qed

\begin{corollary}\label{coro:nu_2=r+1}
Let $r\geq3$ be an odd integer. If $(P,\mathcal{L})$ is an intersecting $r$-uniform linear system with $\tau=r$, then $\tau\leq\displaystyle\frac{|P|+|\mathcal{L}|}{r+1}$.
\end{corollary}
\textit{Proof.} By Lemma \ref{lemma:Delta=r} and Lemma \ref{lemma:nu_2=r+1} follows that $\Delta=\nu_2-1$. Hence, by Theorem \ref{thm:main} and $\nu_2=r+1$ we have $\tau\leq\displaystyle\frac{|P|+|\mathcal{L}|}{r+1}$.\qed

Now, let's deal the case when $r$ is an even integer. If $(P,\mathcal{L})$ is an intersecting $r$-uniform linear system, then $\nu_2\leq r+1$, but, if $r$ is an even integer, then by the lemma given next (given in \cite{Avila2}), Lemma \ref{lemma:impar}, it satisfies	 $\nu_2\leq r$.

\begin{lemma}\cite{Avila2}\label{lemma:impar}
Let $(P,\mathcal{L})$ be an $r$-uniform intersecting linear system with $r\geq2$ be an even integer. If $\nu_2=r+1$ then $\tau=\frac{r+2}{2}$. 
\end{lemma}

\begin{corollary}
Let $(P,\mathcal{L})$ be an $r$-uniform intersecting linear system with $r\geq2$ be an even integer. If $\nu_2=r+1$ then $\tau\leq\displaystyle\frac{|P|+|\mathcal{L}|}{r+1}$. 	
\end{corollary}
\textit{Proof.} It is not difficult to prove that $\Delta=2$, see \cite{Avila2}. Hence, by Corollary \ref{coro:principal_2} we have $\tau\leq\displaystyle\frac{|P|+|\mathcal{L}|}{r+1}$.\qed

Hence, if $\tau>\frac{r+2}{2}$ then $\nu_2\leq r$ and $r\geq4$ an even integer.

\begin{lemma}\label{lemma:nu_2=r}
Let $(P,\mathcal{L})$ be an intersecting $r$-uniform linear system, with $r\geq4$ be an even integer. If $\tau=r$, then $\nu_2=r$.
\end{lemma}

\textit{Proof.}\cite{Avila2}
Let $(P,\mathcal{L})$ be an intersecting $r$-uniform linear system , with $r\geq6$ be an even integer, then $\nu_2\leq r+1$. If $\nu_2=r+1$, then by Lemma \ref{lemma:impar} it follows that $\tau=\frac{r+2}{2}$, which is a contradiction, since $\tau=r$. Therefore $\nu_2\leq r$.
	
Let $p,q\in P$ be two points such that $deg(p)=\Delta$ and $\Delta^\prime=deg(q)=\max\{deg(x): x\in
P\setminus\{p\}\}$. By Theorem \ref{thm:casicucarachageneral}, if $|\mathcal{L}|\leq\Delta+\Delta^\prime+\nu_2-3\leq3(r-1)$, then $\tau\leq\nu_2-1$, which implies $r+1\leq\nu_2$, a contradiction, since $\nu_2\leq r$. Hence, as $\tau\geq\nu_2$ then $|\mathcal{L}|\geq\Delta+\Delta^\prime+\nu_2-2\geq3(r-1)+1$, which implies $\nu_2\geq r$. Hence $\nu_2=r$.\qed
	
\begin{corollary}\label{coro:nu_2=r}
Let $r\geq4$ be an even integer. If $(P,\mathcal{L})$ is an intersecting $r$-uniform linear system with $\tau=r$, then $\tau\leq\displaystyle\frac{|P|+|\mathcal{L}|}{r+1}$.
\end{corollary}
\textit{Proof.} By Lemma \ref{lemma:Delta=r} and Lemma \ref{lemma:nu_2=r} follows that $\Delta=\nu_2=r$. Hence, by Corollary \ref{coro:main_1} we have $\tau\leq\displaystyle\frac{|P|+|\mathcal{L}|}{r+1}$.\qed

By Corollary \ref{coro:nu_2=r} and Corollary \ref{coro:nu_2=r+1}, we have

\begin{theorem}
Let $r\geq3$ be an integer. If $(P,\mathcal{L})$ is an intersecting $r$-uniform linear system with $\tau=r$, then $\tau\leq\displaystyle\frac{|P|+|\mathcal{L}|}{r+1}$.	
\end{theorem}
	
A \emph{finite projective plane}, or merely \emph{projective plane}, is an intersecting linear system satisfying:
\begin{itemize}
	\item any pair of points have a common line,
	\item any pair of lines have a common point, and
	\item  there exist four points in general
	position (there are not three collinear points).
\end{itemize}
It is well known, if $(P,\mathcal{L})$ is a projective plane, there exists a number $q\in\mathbb{N}$, called \emph{order of projective plane}, such that every point (line, respectively) of $(P,\mathcal{L})$ is incident to exactly $q+1$ lines (points, respectively), and $(P,\mathcal{L})$ contains exactly $q^2+q+1$ points (lines, respectively). Also, it is well known that projective planes of order $q$, denoted by $\Pi_q$, exist when $q$ is a power prime. For more information about the existence and the unicity of projective planes see, for instance, \cite{B86,B95}.

In relation to the transversal number of projective planes, it is well known that every line in $\Pi_q$ is a minimum transversal, then $\tau(\Pi_q)=q+1$. On the other hand, related to the 2-packing number of a projective planes, since
projective planes are dual systems, this parameter coincides with
the cardinality of an \emph{oval}, which is the maximum number of
points in general position (no three of them collinear), and it is
equal to $q+1$ when $q$ is odd integer, and it is equal to $q+2$ when $q$ is even integer  (see for example \cite{B95}).

Consequently, for projective planes $\Pi_q$ of odd order $q$ we
have that $\tau (\Pi_q)=\nu_2(\Pi_q)=q+1$. On the other hand, for projective planes $\Pi_q$ of even order $q$ we have that $\tau (\Pi_q)=\nu_2(\Pi_q)-1=q$, see \cite{MR3727901}.

\begin{corollary}
	Let $q$ be a prime power, and let $(P,\mathcal{L})$ be an $(q+1)$-uniform linear system such that $(P,\mathcal{L})\simeq\Pi_q$, then $\tau\leq\displaystyle\frac{|P|+|\mathcal{L}|}{q+2}$.
\end{corollary}
\textit{Proof.} The proof is a simple consequence of Theorem \ref{thm:main} and Corollary \ref{coro:main_1}, since $\Delta=q+1$ and $\nu_2=q+2$, if $q$ is an even integer, and $\nu_2=q+1$, if $q$ is odd integer.

\section{$r$-uniform linear systems with $\nu_2\in\{2,3,4\}$}\label{sec:nu_2}
Let $(P,\mathcal{L})$ be an $r$-uniform linear system with $\nu_2\in\{2,3\}$. It is not difficult to prove (see \cite{MR3727901}) that, $\nu_2=2$, if and only if, $\tau=1$. As $|P|=\Delta(r-1)+1$ and $|\mathcal{L}|=\Delta$, then $\tau<\displaystyle\frac{|P|+|\mathcal{L}|}{r+1}$. On the other hand, if $\nu_2=3$, then $\tau=2$, see \cite{MR3727901}. As $|P|\geq\Delta(r-1)+1$ and $|\mathcal{L}|=\Delta+1$, then $\tau\leq\displaystyle\frac{|P|+|\mathcal{L}|}{r+1}$. 

The main result of this section states the following:
\begin{theorem}\label{thm:main_4}
Let $r\geq2$ be an integer and $(P,\mathcal{L})$ be an $r$-uniform linear system with $|\mathcal{L}|>\nu_2$. If $\nu_2=4$, then $\tau\leq\displaystyle\frac{|P|+|\mathcal{L}|}{r+1}$.  
\end{theorem}

To proove the Theorem \ref{thm:main_4}, we assume $\Delta\leq4$, since

\begin{lemma}\cite{MR3727901}\label{lemma:delta_5}
Any linear system $(P,\mathcal{L})$ with $\nu_2=4$ and $\Delta\geq5$ satisfies $\tau\leq\nu_2-1$. 
\end{lemma}

And by Theorem \ref{thm:main}, we have $\tau\leq\displaystyle\frac{|P|+|\mathcal{L}|}{r+1}$.

\begin{lemma}\cite{MR3727901}\label{lema:tau=3,D=3}
Let $(P,\mathcal{L})$ be a linear system with $\nu_2=4$ and
$\Delta=3$. If $(P,\mathcal{L})\simeq\mathcal{C}_{3,4}$, then $\tau=\nu_2$, otherwise $\tau\leq\nu_2-1$. 
\end{lemma}

\begin{corollary}\label{coro;Delta=3}
Let $r\geq2$ be an integer and $(P,\mathcal{L})$ be an $r$-uniform linear system. If $\nu_2=4$ and $\Delta=3$, then $\tau\leq\displaystyle\frac{|P|+|\mathcal{L}|}{r+1}$. And, the equality holds, if and only if, $(P,\mathcal{L})=\mathcal{C}_{3,4}$. 	
\end{corollary}
\textit{Proof.}
 Let $(P,\mathcal{L})$ be an $r$-uniform linear system with $\nu_2=4$ and $\Delta=3$ such that $(P,\mathcal{L})\not\simeq\mathcal{C}_{3,4}$. By Theorem \ref{thm:main} and Lemma \ref{lema:tau=3,D=3}, we have $\tau\leq\displaystyle\frac{|P|+|\mathcal{L}|}{r+1}$. On the other hand, if $(P,\mathcal{L})\simeq\mathcal{C}_{3,4}$ then $|P|=8+8k$ and $|\mathcal{L}|=8$, where $k+3=r$ and $k\geq0$. Hence
\begin{equation*}
\frac{|P|+|\mathcal{L}|}{r+1}=\frac{8(k+2)}{k+4}\geq\frac{16}{4}=4=\tau.
\end{equation*}
Hence, $\tau\leq\displaystyle\frac{|P|+|\mathcal{L}|}{r+1}$. The equality holds, if and only if, $k=0$, that is, if and only if, $(P,\mathcal{L})=\mathcal{C}_{3,4}$\qed	

\begin{lemma}\cite{MR3727901}\label{lema:tau=3,D=4}
If $(P,\mathcal{L})$ is a linear system with $\nu_2=4$ and $\Delta=4$, then $\tau=\nu_2$.
\end{lemma}

\begin{corollary}
Let $r\geq2$ be an integer and $(P,\mathcal{L})$ be an $r$-uniform linear system. If $\nu_2=4$ and $\Delta=4$, then $\tau\leq\displaystyle\frac{|P|+|\mathcal{L}|}{r+1}$.	
\end{corollary}
\textit{Proof.}
The proof is a direct consequence of Corollary \ref{coro:main_1} and Lemma \ref{lema:tau=3,D=4}. \qed	
\section{$r$-uniform linear systems with $\Delta=2$}\label{sec:Delta=2}
In this section, we present some results of $r$-uniform linear systems $(P,\mathcal{L})$ with $\Delta=2$ satisfying $\tau\leq\displaystyle\frac{|P|+|\mathcal{L}|}{r+1}$. 

\begin{proposition}\label{prop:Delta=2}
	If $(P,\mathcal{L})$ is a linear system with $\Delta=2$, then $\tau\leq\nu_2-1$.
\end{proposition}

\textit{Proof.}
Let $A$ be a maximum subset of $P$ such that every $p\in A$ satisfies $deg(p)=2$, and $\{p,q\}\not\subseteq l$, for every $p,q\in A$. Since $\Delta=2$, then $A\neq\emptyset$. Let $\mathcal{L}_A=\bigcup_{p\in A}\mathcal{L}_p$ and  $\mathcal{L}^\prime=\mathcal{L}\setminus\mathcal{L}_A$. Hence, if $\mathcal{L}^\prime\neq\emptyset$, then the set of lines of $\mathcal{L}^\prime$ is pairwise disjoint. Therefore, the following set $T=A\cup B$, where $B=\{p_l:l\in\mathcal{L}^\prime\mbox{ and }p_l\in l\}$, is a transversal of $(P,\mathcal{L})$. Hence, $\tau\leq|T|=|A|+|B|\leq|\mathcal{L}|-1=\nu_2-1$.\qed

\begin{corollary}
	Let $(P,\mathcal{L})$ be a linear system with $\Delta=2$ and let $\mathcal{L}^\prime$ as above. If $|\mathcal{L}^\prime|\leq1$, then $\tau=\lceil \nu_{2}/2\rceil$. Moreover, if $|\mathcal{L}^\prime|=\nu_2-2$, then $\tau=\nu_2-1$.
\end{corollary}

\begin{corollary}\label{thm:Delta=2}
	If $(P,\mathcal{L})$ is an $r$-uniform linear system with $\Delta=2$ and $\nu_2\geq4$, then $$\lceil \nu_{2}/2\rceil\leq\left\lfloor\displaystyle\frac{|P|+|\mathcal{L}|}{r+1}\right\rfloor\leq\nu_2-1.$$
\end{corollary}
\textit{Proof.}
Let $A$ as in the Proof of Proposition \ref{prop:Delta=2}. If $|A|=k$, where $1\leq k\leq\nu_2(\nu_2-1)/2$, then $ r\nu_2-k\leq|P|\leq r\nu_2-1$. Hence
$$\left\lfloor\frac{|P|+|\mathcal{L}|}{r+1}\right\rfloor\leq
\left\lfloor\nu_2-\frac{1}{r+1}\right\rfloor=\nu_2-1.$$

\

On the other hand, since $|P|\geq r\nu_2-k$ then $$\left\lfloor\frac{|P|+|\mathcal{L}|}{r+1}\right\rfloor\geq
\left\lfloor\nu_2-\frac{k}{r+1}\right\rfloor\geq\left\lfloor\nu_2-\frac{\nu_2(\nu_2-1)/2}{r+1}\right\rfloor\geq\lceil \nu_{2}/2\rceil,$$and the statement holds.\qed

In \cite{Dorfling}, was proved the following: 

\begin{theorem}\cite{Dorfling}
	If $(P,\mathcal{L})$ is a linear system with $\Delta=2$, then $$\tau\leq\displaystyle\frac{|P|+|\mathcal{L}|}{r+1}.$$ 	
\end{theorem}

As a simple consequence, since $\tau\in\mathbb{N}$, then $\tau\leq\left\lfloor\frac{|P|+|\mathcal{L}|}{r+1}\right\rfloor$.

\begin{theorem}
If $(P,\mathcal{L})$ is an $r$-uniform linear system with $\nu_2-1\leq r$, then $$\lceil \nu_{2}/2\rceil\leq\frac{|P|+|\mathcal{L}|}{r+1}.$$
\end{theorem}
\textit{Proof.}
Since $|P|\geq r\nu_2-\frac{\nu_2(\nu_2-1)}{2}$ and $|\mathcal{L}|\geq\Delta+\nu_2-2$, then
\begin{eqnarray*}
	\frac{|P|+|\mathcal{L}|}{r+1}&\geq&\frac{r\nu_2-\frac{\nu_2(\nu_2-1)}{2}+\nu_2+\Delta-2}{r+1}\\
	&=&\nu_2\left[1-\frac{\nu_2-1}{2(r+1)}\right]+\frac{\Delta-2}{r+1}
\end{eqnarray*} 
Since $\nu_2-1\leq r$ then
\begin{eqnarray*}
	\frac{|P|+|\mathcal{L}|}{r+1}&\geq&\nu_2\left[1-\frac{\nu_2-1}{2\nu_2}\right]+\frac{\Delta-2}{r+1}\\
	&=&\frac{\nu_2+1}{2}+\frac{\Delta-2}{r+1}\\
	&\geq&\lceil \nu_{2}/2\rceil.
\end{eqnarray*}
And the theorem holds.\qed	

\begin{corollary}
	Let $(P,\mathcal{L})$ be an $r$-uniform linear system with $\nu_2-1\leq r$ and $\tau=\lceil\nu_{2}/2\rceil$, then $\tau\leq\displaystyle\frac{|P|+|\mathcal{L}|}{r+1}$.
\end{corollary}

\begin{corollary}\label{coro:principal_2}
	Let $(P,\mathcal{L})$ be an $r$-uniform intersecting linear system with $\Delta=2$, then $\tau\leq\displaystyle
	\frac{|P|+|\mathcal{L}|}{r+1}$.
\end{corollary}

\begin{corollary}
	Let $(P,\mathcal{L})$ be an $r$-uniform intersecting linear system with $\Delta=2$ and $r\geq2$ an even integer. Then $\tau=\displaystyle
	\frac{|P|+|\mathcal{L}|}{r+1}$, if and only if, $r=\nu_2-1$.
\end{corollary}

In \cite{Dorfling}, was proved the following
\begin{theorem}\cite{Dorfling}
	If $(P,\mathcal{L})$ is a linear system with $\Delta=2$ and 
	$\tau=\displaystyle\frac{|P|+|\mathcal{L}|}{r+1}$, then $(P,\mathcal{L})$ is an $(\nu_2-1)$-uniform linear system.	
\end{theorem} 
\section*{Acknowledgement} 
Research supported by SNI and CONACyT.

\end{document}